\documentclass[11pt]{article}

\usepackage{fullpage}
\usepackage{amsthm}
\usepackage{amssymb}
\usepackage{amsmath}
\usepackage{graphicx}
\usepackage{tikz}
\usepackage{mathrsfs}
\usepackage{float}
\usepackage{makecell}
\usepackage{amscd}
\usepackage{multirow}
\usepackage{cite}
\usepackage{verbatim}
\usepackage{hhline}
\usepackage{comment}

\usepackage{pgfplots}
\pgfplotsset{compat=1.15}
\usetikzlibrary{arrows}

\renewcommand{\epsilon}{\varepsilon}
\renewcommand{\diamond}{\diamondsuit}
\DeclareMathOperator{\lk}{lk}

\DeclareMathOperator{\ex}{ex}
\usepackage{relsize}
\newtheorem{theorem}{Theorem}
\newtheorem{lemma}[theorem]{Lemma}

\newtheorem{proposition}[theorem]{Proposition}

\newtheorem{conjecture}[theorem]{Conjecture}
\newtheorem{definition}[theorem]{Definition}
\newtheorem{question}[theorem]{Question}

\title{A conditional lower bound for the Tur\'an number of spheres}
\author{Andrew Newman\thanks{Carnegie Mellon University}, Marta Pavelka\thanks{Carnegie Mellon University}}

\begin{document}
\maketitle
\begin{abstract}
We consider the hypergraph Tur\'an problem of determining $\ex(n, S^d)$, the maximum number of facets in a $d$-dimensional simplicial complex on $n$ vertices that does not contain a simplicial $d$-sphere (a \emph{homeomorph} of $S^d$) as a subcomplex. We show that if there is an affirmative answer to a question of Gromov about sphere enumeration in high dimensions, then $\ex(n, S^d) \geq \Omega(n^{d + 1 - (d + 1)/(2^{d + 1} - 2)})$. Furthermore, this lower bound holds unconditionally for 2-LC spheres, which includes all shellable spheres and therefore all polytopes. We also prove an upper bound on $\ex(n, S^d)$ of $O(n^{d + 1 - 1/2^{d - 1}})$ using a simple induction argument. We conjecture that the upper bound can be improved to match the conditional lower bound.
\end{abstract}

\section{Introduction}
Recall that the Tur\'an number of a graph $H$, denoted $\ex(n, H)$, is the maximum number of edges in an $H$-free graph on $n$ vertices. Establishing bounds for the Tur\'an number of various graphs and families of graphs is at the heart of extremal graph theory. Moreover, there has also been a great deal of work on hypergraph Tur\'an problems, \cite{KeevashSurvey} surveys many of these results. 

Here, we take a topological perspective on a hypergraph Tur\'an problem. As such we will regard $(d + 1)$-uniform hypergraphs as (pure) $d$-dimensional simplicial complexes and study the bounds for \emph{homeomorphs} of spheres. For $\mathcal{H}$ a family of $d$-complexes, we denote by $\ex(n, \mathcal{H})$ the maximum number of $d$-faces in a pure $d$-complex on $n$ vertices which does not contain any complex in $\mathcal{H}$ as a subcomplex. While $\mathcal{H}$ could be any family of complexes, here we will focus on the case that $\mathcal{H}$ is some collection of subcomplexes that are all homeomorphic to a particular fixed topological space. 

Our work is motivated by recent results of Keevash, Long, Narayanan, Scott, and Yap \cite{KeevashLongNarayananScott, LongNarayananYap} who study homeomorph Tur\'an problems for arbitrary complexes, as well as by the work of Kupavskii, Polyanskii, Tomon, Zakharov \cite{KupavskiiPolyanskiiTomonZakharov} and Sankar \cite{Sankar} establishing the correct asymptotics of homeomorphs of arbitrary fixed surfaces in 2-dimensional simplicial complexes. This work on surfaces is a far-reaching extension of the work of Brown, Erd\H{o}s, and S\'os \cite{BrownErdosSos} who initially studied Tur\'an problems for 2-spheres. 

Particularly relevant to our study here is the main result of Long, Narayanan, and Yap \cite{LongNarayananYap} who show that for any $d \geq 1$ there exists a constant $\lambda_d \geq d^{-2d^2}$ such that for any fixed $d$-complex $S$ there is an upper bound of $O(n^{d + 1 - \lambda_d})$ for the Tur\'an number of the family of homeomorphs of $S$. In the conclusion of their paper, they speculate that perhaps the precise value of $\lambda_d$ corresponds with the exponent for the Tur\'an number for homeomorphs of $d$-dimensional spheres (equivalently for simplicial $d$-spheres). They point out though that establishing this exponent for the $d$-sphere itself is an open problem. It is this problem that we study here. 

It is well known that spheres in dimension 3 and higher are much different combinatorially than circles and 2-spheres, and this appears to be a major obstacle for this high-dimensional Tur\'an problem. Letting $\ex(n, S^2)$ denote the maximum number of triangles in a 2-complex on $n$ vertices that contains no homeomorph of $S^2$, Brown, Erd\H{o}s, and S\'os \cite{BrownErdosSos} proved
\[\ex(n, S^2) = \Theta(n^{5/2}).\]

While the original proof for the lower bound in \cite{BrownErdosSos} gave an explicit construction, Linial \cite{LinialSlides} gave a simple randomized construction that extends to any fixed surface. The upper bound of $O(n^{5/2})$ has recently been shown to also hold for any fixed surface \cite{KupavskiiPolyanskiiTomonZakharov, Sankar}. It is conjectured, see \cite{KeevashLongNarayananScott}, that $\Theta(n^{5/2})$ is the correct answer for triangulations of any fixed 2-complex. The best-known upper bound for triangulations of any 2-complex is $O(n^{14/5})$ due to \cite{KeevashLongNarayananScott}.

The key to Linial's randomized construction is the result of Tutte counting labeled 2-spheres on $k$ vertices. Tutte showed that the number of labeled 2-spheres on $k$ vertices is at most $O \left(k! \left( \frac{256}{27} \right)^k \right)$. However, in dimensions 4 and higher, Kalai showed \cite{KalaiSpheres} that there are at least $\exp(\Theta(k^{\lfloor d/2 \rfloor}))$ simplicial $d$-spheres on $k$ labeled vertices, and even for $d = 3$ there is a lower bound of $\exp(\Theta(k^2))$ due to Nevo, Santos, and Wilson \cite{NevoSantosWilson}. These bounds on the number of $d$-spheres for $d > 2$ suggest that a different approach is needed to give a lower bound for $\ex(n, S^d)$, the maximum number of facets in a $d$-complex on $n$ vertices with no homeomorph of $S^d$.

However, if we view Tutte's bound from the point of view of the number of \emph{facets}, we arrive at an interesting open problem. In the 2-dimensional case, the $k!$ in Tutte's result comes from labeling the vertices, and by the Euler characteristic, the number of \emph{unlabeled} 2-spheres with $m$ facets is at most $O \left( \left(\frac{16}{3\sqrt{3}} \right)^m \right)$.  By trying to generalize this type of bound, we get to an important question popularized by Gromov \cite{GromovSpaces}: For $d \geq 3$ is there a constant $C_d$ such that the number of unlabeled $d$-spheres with $m$ facets is at most $C_d^m$?
For the importance of this question in physics, see \cite{PonzanoRegge, Regge, ReggeWilliams}, there has been a search for classes of triangulations of spheres of exponential size in terms of the number of facets since the sixties. 

Some results in this direction include the work of Durhuus and Jónsson, who proved that the class of LC (locally constructible) 3-spheres has an exponential size \cite{DurhuusJonsson}. Benedetti and Ziegler generalized this result to higher dimensions and showed that LC $d$-spheres with $m$ facets are less than $2^{d^2 \, m}$ \cite{BenedettiZiegler}. This was, in particular, the first proof that the number of simplicial polytopes is exponential, since the boundaries of simplicial polytopes are shellable spheres and all shellable spheres are LC \cite{BenedettiZiegler}. Benedetti and the second author defined the class of 2-LC spheres that generalizes the LC notion. They proved that this broader class also has an exponential size. Another exponentially large class of triangulated $d$-spheres are triangulations with bounded `discrete Morse vector' by Benedetti \cite{BenedettiMorse}. This class contains the LC class. However, there is no known characterization of the 2-LC class in terms of the discrete Morse vector.

Our main theorem will hold for classes of spheres like these which are ``only exponentially large", so to state our main result, we introduce the following definition.

\begin{definition}
    A class $\mathcal{S}$ of unlabeled simplicial $d$-spheres has \emph{exponential size} in terms of the number of facets $m$ if there exists a constant $C = C(\mathcal{S}, d)$ such that for all $m$ the number of $d$-spheres in $\mathcal{S}$ with $m$ facets is less than $C^m$.
\end{definition}

Our main theorem is as follows.
\begin{theorem}\label{maintheorem}
Fix $d \geq 3$, and let $\mathcal{S}$ be a collection of unlabeled $d$-spheres that has exponential size, then 
\[\ex(n, \mathcal{S}) = \Omega \left(n^{d + 1 - (d + 1)/(2^{d + 1} - 2)}\right).\]
\end{theorem}
This theorem implies that if there is an affirmative answer to Gromov's question then $\ex(n, S^d) = \Omega \left(n^{d + 1 - (d + 1)/(2^{d + 1} - 2)}\right)$. Conversely, if our lower bound does not actually hold for the class of all triangulations of the $d$-sphere then a proof of that fact would require finding an exceptionally complicated triangulated sphere, as all classes which are known to have exponential size include 2-LC spheres with in turn include shellable spheres. For convenience we restrict our focus to spheres; however, Theorem \ref{maintheorem} holds if we replace $d$-spheres with $d$-manifolds, as we further comment on in the conclusion.

For comparison with our lower bound, we also give a simple inductive proof for an upper bound; the $d = 2$ case is in the original paper of Brown, Erd\H{o}s, and S\'os \cite{BrownErdosSos}. 
\begin{proposition}\label{upperboundtheorem}
For $d \geq 2$ the Tur\'an number for homeomorphs of $S^d$ satisfies
\[\ex(n, S^d) \leq O(n^{d + 1 - 1/2^{d - 1}}).\]
\end{proposition}

\section{Preliminaries}
Homeomorphs of $S^1$ are just cycle graphs and so for this reason the Tur\'an number for simplicial 1-spheres is trivially $n - 1$. For homeomorphs of $S^2$, Brown, Erd\H{o}s, and S\'os \cite{BrownErdosSos} give a double counting argument to show the $O(n^{5/2})$ upper bound on $\ex(n, S^2)$. We induct on this argument to prove Proposition \ref{upperboundtheorem}. For our lower bound, we follow the strategy of Linial's proof of $\ex(n, S^2) \geq \Omega(n^{5/2})$. Let us start with a review of that argument. A fundamental part of Linial's randomized construction is an enumeration result of Tutte.

\begin{lemma}[Tutte \cite{TutteCensus}]
The number of simplicial $3$-polytopes with $k$ labeled vertices is asymptotically $O \left(k! \left( \frac{256}{27} \right)^k \right)$. 
\end{lemma}

Moreover, it is well known by Steinitz' Theorem that \emph{every} simplicial 2-sphere is the boundary of a 3-polytope, see Chapter 4 of \cite{Ziegler1995}. From here, the proof for the lower bound for $\ex(n, S^2)$ follows by a standard probabilistic argument. Take $Y_d(n, p)$ to denote the Linial--Mesulam random $d$-complex model, i.e., each face of dimension $(d - 1)$ is present, and the $d$-faces are included independently with probability $p$. Recall that for random simplicial complexes (including random graphs), a property is said to hold with high probability if the property holds with probability tending to one as the number of vertices goes to infinity. 

For the question of the extermal construction for the 2-sphere one takes $Y \sim Y_2(n, p)$ with $p = {\epsilon}/{\sqrt{n}}$, $\epsilon > 0$ a small constant. As the number of triangles in $Y$ follows a binomial distribution with $\binom{n}{3}$ trials and success probability $p$, it follows that for $p = \epsilon/\sqrt{n}$ with high probability the number of triangles in $Y_2(n, p)$ is $\Theta(n^{5/2})$. 

On the other hand, by the Euler characteristic a simplicial 2-sphere with $k$ vertices has $2k - 4$ triangles, so by the linearity of expectation and Tutte's result the expected number of embedded 2-spheres is at most
\begin{eqnarray*}
\sum_{k = 4}^{\infty} \binom{n}{k} k! C^k \left(\frac{\epsilon}{\sqrt{n}} \right)^{2k - 4} &\leq& O(n^2),
\end{eqnarray*}
for $\epsilon < 1/C$. With high probability, there are therefore $o(n^{5/2})$ embedded 2-spheres, and after simply deleting a triangle from each one, there are still $\Theta(n^{5/2})$ triangles, but no 2-spheres.

Our strategy to prove Theorem \ref{maintheorem} is to adapt this randomized argument to higher dimensions. In place of Tutte's result, we assume that we work with classes of spheres which has exponential size in terms of the number of facets. As discussed in the introduction, this is a reasonable assumption as there are many natural classes of simplicial $d$-spheres for which such a bound provably holds, and, as far as current knowledge goes, such a bound may hold for all $d$-spheres. As discussed in the introduction, the classes of spheres for which it is known to have exponential size include all simplicial spheres in the following hierarchy: 
\[\{\text{polytopal spheres}\} \subseteq \{\text{shellable spheres}\} \subseteq \{\text{LC spheres}\} \subseteq \{\text{2-LC spheres}\},\]
as well as spheres that have bounded discrete Morse vector.

Since we do not know whether the collection of all $d$-spheres has exponential size, it could be that the true value of $\ex(n, S^d)$ is smaller than our conditional lower bound from Theorem \ref{maintheorem}. However, if that is the case, a proof of a better upper bound would have to be exceptionally difficult, as it would necessarily involve finding a simplicial sphere that does not belong to any of the classes listed above. On the other hand, the proof of our upper bound in Proposition \ref{upperboundtheorem} finds an iterated suspension of a cycle in every sufficiently dense $d$-complex, which in particular is a polytopal sphere. 

For context regarding the common types of spheres for which our lower bound holds unconditionally, we review some standard definitions. 

\begin{definition}
    A \emph{polytopal $d$-sphere} is the boundary of a simplicial $(d + 1)$-dimensional polytope.
\end{definition}

\begin{definition}
    A $d$-dimensional simplicial complex $X$ is \emph{shellable} if it is pure $d$-dimensional and there is an ordering $\sigma_1, ..., \sigma_m$ of all the facets of $X$ so that for all $2 \leq k \leq m$, $(\sigma_1 \cup \cdots \cup \sigma_{k - 1}) \cap \sigma_k$ is pure $(d - 1)$-dimensional.
\end{definition}

All polytopes are shellable \cite{BruggesserMani}, but not all shellable spheres are polytopal and not all simplicial spheres are shellable; see the discussion in Chapter 8 of \cite{Ziegler1995}. With respect to counting by \emph{vertices} there is an upper bound of Goodman and Pollack of $n^{(d+1)(d+2)n}$ polytopal $d$-spheres on $n$ labeled vertices, while Kalai showed there are at least $\exp(\Theta(n^{\lfloor d/2 \rfloor}))$ simplicial $d$-spheres on $n$ labeled vertices \cite{KalaiSpheres}; Lee \cite{Lee} showed that the spheres constructed in \cite{KalaiSpheres} are shellable. 

Broadening out from shellable spheres we have LC-spheres introduced by Durhuus and J\'onsson \cite{DurhuusJonsson}, and more generally $t$-LC spheres introduced by Benedetti and the second author \cite{BenedettiPavelka}. Of particular interest here is the case $t = 2$ as these are the largest class in the $t$-LC heirarchy for which there is a proof of exponential size.

\begin{definition}
A $d$-dimensional simplicial sphere is \emph{$2$-LC} if it is obtainable from a tree of $d$-simplices by recursively identifying two boundary $(d-1)$-faces whose intersection has dimension at least $d-3$.
\end{definition}

Regarding enumeration we have the following bound, which holds also for all 2-LC manifolds. 
\begin{theorem}[Benedetti--Pavelka \cite{BenedettiPavelka}] For any $d \geq 3$, the number of 2-LC $d$-spheres with $m$ facets, for $m$ large, is smaller than $ 2^{\frac{d^3}{2}m}$.
\end{theorem}

Beyond the exponential-size assumption, another key result in our proof is a lower bound theorem of Goff, Klee, and Novik for balanced spheres. Recall that a simplicial $d$-sphere, or more generally a $d$-complex, is balanced if it is possible to properly color the underlying graph with exactly $d + 1$ colors. In exactly the same way that Tur\'an numbers for non-bipartite graphs are always quadratic in $n$, $\ex(n, X) = \Theta(n^{d + 1})$ for any $d$-complex $X$ which is not balanced as we can simply take the complete $(d + 1)$-partite hypergraph with $n/(d + 1)$ vertices in each part and have an $X$-free $d$-complex with $\Omega(n^{d + 1})$ facets. If we can construct a complex $Y$ that contains no balanced triangulation of $S^d$ then by coloring each vertex of $Y$ uniformly at random from a set of $d + 1$ colors and keeping only the rainbow facets, we end up with a complex $Y'$ that has no $d$-spheres at all and $f_d(Y')/f_d(Y)$ very close to $\frac{(d + 1)!}{(d + 1)^{d + 1}}$ with high probability, so we lose only a constant factor in the number of facets.

\section{Upper bound}
The proof for the upper bound is a straightforward induction from the standard proof for the upper bound in the $d = 2$ case; however, it does not seem to appear in the literature, so we include it here. For this upper bound, we recall the definition of a link in a simplicial complex. For $X$ a simplicial complex and $\sigma$ a face of $X$, the link of $\sigma$ in $X$ is defined as 
\[\lk(\sigma) = \{\tau \in X \mid \sigma \cap \tau = \emptyset \text{ and } \sigma \cup \tau \in X \}.\]

\begin{proof}[Proof of Proposition \ref{upperboundtheorem}]
By induction on $d$. The $d = 1$ case is trivial, as it is just the Tur\'an number for a cycle which is exactly $n - 1$. For $d > 1$ suppose that we have a $d$-complex with at least $C n^{d + 1 - 1/2^{d - 1}}$ facets for some large constant $C$ that depends on $d$. Now for $u$ and $v$ vertices, let $d(u, v)$ denote the number of $(d - 1)$ faces in $\lk(u) \cap \lk(v)$ and take the following sum.
\begin{eqnarray*}
\sum_{v, u} d(u, v) &=& \sum_{\sigma \in \mathcal{F}_{d - 1}} \binom{\deg(\sigma)}{2},
\end{eqnarray*}
where $\mathcal{F}_{i}$ denotes the set of $i$-dimensional faces and the degree of a face is the number of faces of one dimension larger that contain it. We continue by estimating the above expression from below using convexity.
\begin{eqnarray*}
\sum_{\sigma \in \mathcal{F}_{d - 1}} \binom{\deg(\sigma)}{2} &\geq& \sum_{\sigma \in \mathcal{F}_{d - 1}} \frac{\deg(\sigma)^2}{4} \\
&\geq& \frac{1}{4} \frac{ \left(\sum_{\sigma \in \mathcal{F}_{d - 1}} \deg(\sigma)\right)^2}{f_{d - 1}} \\
&\geq& \frac{1}{4} \frac{ ((d + 1) f_d)^2 }{n^d} \\
&\geq& \Omega \left( \frac{n^{2d + 2 - \frac{1}{2^d - 2}}}{n^d} \right) \\
&=& \Omega \left(n^{d + 2 - \frac{1}{2^{d - 2}}}\right) .
\end{eqnarray*}
As the sum was taken over all pairs of vertices this means that there is a pair of vertices with at least $\Omega(n^{d - 1/2^{d - 2}})$-many $(d - 1)$-faces in the common link. By induction this common link contains a $(d - 1)$-sphere and so we have found a $d$-sphere as the suspension of a $(d - 1)$-sphere.
\end{proof}

\section{Lower bound}

Here we prove Theorem \ref{maintheorem}. As discussed we will use a lower bound theorem for balanced spheres due to Goff, Klee, and Novik:
\begin{theorem}[From Theorem 5.3 of Goff--Klee--Novik \cite{GoffKleeNovik}] \label{GKNLowerBound}
Let $X$ be a balanced, simplicial $d$-sphere with $n$ vertices. Then 
\[f_d(X) \geq \frac{2^{d + 1} - 2}{d + 1} n - 2^{d + 1} + 4.\]
\end{theorem}

With this result, we prove the following enumeration lemma that will be critical for our randomized construction. For a collection $\mathcal{S}$ of $d$-dimensional simplicial complexes we denote by $\mathcal{S}_m$ the subcollection of $\mathcal{S}$ of complexes having exactly $m$ facets.
\begin{lemma}\label{keylemma}
Let $\mathcal{S}$ be a collection of unlabeled, balanced simplicial $d$-spheres that has exponential size, say $|\mathcal{S}_m| \leq C^m$. Then the number of labeled copies of elements of $\mathcal{S}_m$ in the simplex $\Delta_{n - 1}$ is at most
\[C^m n^{(d + 1)m/(2^{d + 1} - 2)} n^{(d + 1)(2^{d + 1} - 4)/(2^{d + 1} - 2)}.\]
\end{lemma}
\begin{proof}
For an unlabeled, balanced simplicial $d$-sphere $M \in \mathcal{S}_m$ the number of labeled copies of $M$ in $\Delta_{n - 1}$ is at most the number of ways to label the vertices of $M$ with elements from $[n]$. The number of such labelings is at most
\[n^{f_0(M)}.\]
By Theorem \ref{GKNLowerBound}, 
\[m = f_d(M) \geq \frac{2^{d + 1} - 2}{d + 1} f_0(M) - 2^{d + 1} + 4\]
and so 
\[f_0(M) \leq \frac{(m + 2^{d + 1} - 4)(d + 1)}{2^{d + 1} - 2}. \]
Thus,
\[n^{f_0(M)} \leq n^{\frac{(m + 2^{d + 1} - 4)(d + 1)}{2^{d + 1} - 2}}.\]
The result follows as there are at most $C^m$ choices for $M$.
\end{proof}

We can now prove Theorem \ref{maintheorem} via a standard alterations argument.
\begin{proof}[Proof of Theorem \ref{maintheorem}]
 Let $\mathcal{S}$ be a family of unlabeled $d$-spheres that has exponential size. Take $Y \sim Y_d(n, \epsilon n^{-(d + 1)/(2^{d + 1} - 2)})$ for $\epsilon$ a small constant to be set later. By a usual application of the Chernoff bound, with high probability $Y$ has at least $\Omega(n^{d + 1 - (d + 1)/(2^{d + 1} - 2)})$ facets. By Lemma \ref{keylemma}, the expected number of \emph{balanced} complexes in $\mathcal{S}$ that are included in $Y$ is at most
\begin{eqnarray*}
&&\sum_{m = 2^{d + 1}}^{\infty} C^m n^{(d + 1) m/(2^{d + 1} - 2)} n^{(d + 1)(2^{d + 1} - 4)/(2^{d + 1} - 2)} \left(\frac{\epsilon}{n^{(d + 1)/(2^{d + 1} - 2)}} \right)^m \\
&=& n^{(d + 1)(2^{d + 1} - 4)/(2^{d + 1} - 2)}\sum_{m = 2^{d + 1}}^{\infty} (C\epsilon)^m .
\end{eqnarray*}
For $\epsilon < 1/C$ the above is $O(n^{(d + 1)(2^{d + 1} - 4)/(2^{d + 1} - 2)})$. With high probability, the expected number of embedded balanced complexes from $\mathcal{S}$ in $Y$ is $o(n^{d + 1 - (d + 1)/(2^{d + 1} - 2)})$. Thus, with high probability, we may remove a $o(1)$ fraction of the facets from $Y$ to end up with a complex $\widetilde{Y}$ still with at least $\Omega(n^{d + 1 - (d + 1)/(2^{d + 1} - 2)})$ facets, but without balanced members of $\mathcal{S}$. For the unbalanced members of $\mathcal{S}$ we simply color each vertex of $[n]$ independently and uniformly from the set of $d + 1$ colors. Now, let $Z$ be the subcomplex of $\widetilde{Y}$ generated by the rainbow facets under this coloring. With high probability, $Z$ will contain about a $\frac{(d + 1)!}{(d + 1)^{d + 1}}$ fraction of the facets of $\widetilde{Y}$, but it will not contain any complexes of $\mathcal{S}$.
\end{proof}

\section{Concluding remarks}
The main result here is a conditional lower bound for $\ex(n, S^d)$ and an unconditional lower bound restricted to, for example, polytopal, shellable, or locally constructible spheres. As unbalanced triangulations are easy to avoid, our first open question here is:
\begin{question}
    Does the class of balanced simplicial $d$-spheres have exponential size? If so, then our lower bound holds for all simplicial spheres.
\end{question}

We also mention that our result extends to manifolds. This follows from a result of Juhnke-Kubitzke, Murai, Novik, and Sawaske \cite{JKMNS} that extended the lower bound in Theorem \ref{GKNLowerBound} to manifolds (in fact to homology manifolds). Moreover, the 2-LC manifolds have exponential size by the main result of \cite{BenedettiPavelka}.

Next, we can ask whether our conditional lower bound can be improved or if one should expect a matching upper bound.

\begin{question}
    Is $\ex(n, S^d) \leq O(n^{d + 1 - (d + 1)/(2^{d + 1} - 2)})$?
\end{question}

Regarding this question, we conjecture that the answer is yes. Moreover, \cite{KeevashLongNarayananScott, LongNarayananYap} studied the question of universal exponents for homeomorphs. They define $\lambda_d$ as the minimum value exponent so that for any fixed $d$-complex $S$ the collection $\mathcal{H}(S)$ of homeomorphs of $S$ satisfies:
\[\ex(n, \mathcal{H}(S)) \leq O(n^{d + 1 - \lambda_d}).\]
Note that $\lambda_d$ does not depend on $S$, only on the dimension. Mader \cite{Mader} showed that $\lambda_1 = 1$ and Keevash, Long, Narayanan, and Scott\cite{KeevashLongNarayananScott} proved $\lambda_2 \geq \frac{1}{5}$ and conjectured that $\lambda_2 = \frac{1}{2}$. In higher dimensions, Long, Narayanan, and Yap \cite{LongNarayananYap} showed $\lambda_d \geq \frac{1}{d^{2d^2}}$, but did not make a conjecture about the true value. We have shown that if the collection of all simpicial $d$-spheres has exponential size then $\lambda_d \leq \frac{d + 1}{2^{d + 1} - 2}$. We conjecture that a matching lower bound holds.

\begin{conjecture}
    $\lambda_d \geq \frac{d + 1}{2^{d + 1} - 2}$.
\end{conjecture}

Our reason for believing that this conjectured lower bound holds comes from the sharpness of the lower bound theorem of \cite{GoffKleeNovik} that we use in our proof. Their lower bound theorem is demonstrated to be sharp by connected sums of cross-polytopes. For any starting complex $X$ we can use connected sums with cross-polytopes to find balanced triangulations of $X$ that are asymptotically as sparse as possible. To explain how this triangulation works, we define for $\sigma$ a facet an \emph{octahedral flip of $X$ at $\sigma$} to be the complex $X'$ obtained from $X$ by taking its connected sum with the boundary of the $(d + 1)$-cross-polytope, denoted $\diamond_{d}$, at $\sigma$. This means that one connects $X$ to $\diamond_d$ along a common face $\sigma$ and then deletes $\sigma$. Suppose we triangulate $X$ as follows. Let $X_1$ be the barycentric subdivision of $X$ (so that we start with a balanced complex). From here, let $X_{i + 1}$ be a complex obtained from $X_i$ by performing an octahedral flip at some facet of $X_i$. Every octahedral flip adds $(d + 1)$ vertices and a net total of $2^{d + 1} - 2$ facets. As $i \rightarrow \infty$ the density $f_{d}(X_i)/f_0(X_i)$ therefore tends to $\frac{2^{d + 1} - 2}{d + 1}$, and we can get a balanced triangulation of $X$ that is arbitrarily close to the lower bound theorem for balanced spheres. A potential strategy, at least for showing that $\lambda_d \geq \frac{d + 1}{2^{d + 1} - 2} - \epsilon$, would be to find a subdivision $X_L$ of $X$ for $L = L(\epsilon)$ in a dense enough $d$-complex. Prior work on universal exponents from \cite{KeevashLongNarayananScott, LongNarayananYap} looks for very specific subdivisions, so it seems reasonable to expect that the sparsest subdivision would give the best bound.

\section*{Acknowledgement}
We thank Florian Frick for introducing us to the Tur\'an problem we study here.

\bibliography{ResearchBibliography}
\bibliographystyle{amsplain}
\end{document}